\documentclass[11pt,reqno]{amsart}
\usepackage[T1]{fontenc}
\usepackage[letterpaper,margin=1in]{geometry}
\usepackage{amsmath,amssymb,mathtools}
\usepackage{booktabs}
\usepackage{microtype}
\usepackage[hidelinks]{hyperref}
\hypersetup{pdftitle={A combinatorial proof of the Chan--Robbins--Yuen volume formula},
  pdfauthor={Nurtas Shyntas}}
\newtheorem{theorem}{Theorem}[section]
\newtheorem{proposition}[theorem]{Proposition}
\newtheorem{lemma}[theorem]{Lemma}

\theoremstyle{remark}

\numberwithin{equation}{section}
\newcommand{\A}{\mathcal A}
\newcommand{\F}{\mathcal F}

\newcommand{\indeg}{\operatorname{indeg}}
\newcommand{\sqcupdot}{\sqcup}
\DeclareMathOperator{\vol}{vol}
\title[A combinatorial proof of the CRY volume formula]{A combinatorial proof of the\protect\\ Chan--Robbins--Yuen volume formula}
\author{Nurtas Shyntas}
\address{Independent Researcher}
\date{13 October 2026}

\begin{document}
\begin{abstract}
The Chan--Robbins--Yuen polytope $\mathrm{CRY}_{m+1}$ has normalized volume
$\prod_{j=1}^{m-1}C_j$, a product of Catalan numbers.
Zeilberger proved the formula using the Morris constant-term identity;
finding a combinatorial proof remained an open problem. We give such a
proof by counting nonnegative integer flows. We add an extra edge from
each vertex to a common sink and count flows according to how many of
these edges carry positive flow. A double count proves a special case of
a recurrence of Morales and Shi and gives these counts in terms of
Narayana numbers. Deleting the first vertex then yields the Catalan
product. We also give an explicit bijection for the recurrence
and prove the corresponding reduction formula posed by Postnikov.
\end{abstract}
\maketitle

\section{Introduction}

The Chan--Robbins--Yuen volume formula arose from the study of volumes of
faces of the Birkhoff polytope, the polytope of doubly stochastic matrices.
In a 1998 preprint, later published as~\cite{cry}, Chan, Robbins, and Yuen
conjectured that the normalized volume of one such face is a product of
consecutive Catalan numbers. They verified the formula for matrices of
size at most~12 and gave a triangulation that reduced the conjecture to
counting triangular arrays of nonnegative integers.
Zeilberger gave a proof using the Morris constant-term identity, published
in 1999~\cite{zeilberger}. The problem of finding a combinatorial proof
remained open.

Morales and Shi~\cite{morales-shi} approached this problem through integer
flows. They introduced a refinement of the Morris identity and proved
several of its relations combinatorially. In Section~7.2 of their paper,
they asked for a combinatorial proof of the remaining recurrence, their
equation~(19). The search for a combinatorial proof of the volume formula
is also discussed in Dugan's 2025 paper~\cite{f-vectors}.
Here we give a double-counting proof, and an explicit bijective proof, of
the recurrence for the flow sets defined in Section~\ref{sec:two-exits}.
In Morales and Shi's notation, this is the case $a=b=c=1$.
This specialization suffices to give
a combinatorial proof of the Chan--Robbins--Yuen volume formula.

A nonnegative real flow assigns a nonnegative real number to each edge
of a directed graph. Its net flow at a vertex is the outgoing flow minus
the incoming flow. We use the flow-polytope description of
$\mathrm{CRY}_{m+1}$: it is the set of such flows on the complete directed
graph with vertices $0,1,\ldots,m+1$, all edges directed from smaller to
larger vertices, with net flow $(1,0,\ldots,0,-1)$.
Normalized volume is $d!$ times volume measured in coordinates from the
integer-flow lattice, where $d$ is the dimension. The formula we prove
combinatorially is Zeilberger's
theorem~\cite{zeilberger}.
\begin{theorem}[Chan--Robbins--Yuen volume formula]\label{thm:volume}
For every $m\geq1$,
\begin{equation}\label{eq:volume}
 \vol(\mathrm{CRY}_{m+1})=\prod_{j=1}^{m-1}C_j,
 \qquad C_j=\frac{1}{j+1}\binom{2j}{j}.
\end{equation}
\end{theorem}

A related formulation appears in Postnikov's 2002 list of open problems
\cite[Problem~1(C)--(D)]{postnikov}.
Consider commuting variables $x_{ij}$ indexed by $i<j$ and the reduction rule
\begin{equation}\label{eq:rule}
 x_{ij}x_{jk}\longmapsto x_{ij}x_{ik}+x_{ik}x_{jk}
 \qquad(i<j<k).
\end{equation}
Starting from a monomial $M$, apply this rule to each resulting term until no
further reduction is possible. A complete reduction tree records this
process: its root is $M$, each reduction gives two children, and each leaf
is a terminal monomial, meaning that no reduction applies to it.
Leaves are counted with multiplicity: two occurrences of the same
terminal monomial count separately. The resulting
polynomial can depend on the reductions chosen, but its number of terms does
not. We denote this number by $N(M)$; Lemma~\ref{lem:reduction} gives a direct
proof that it is well defined.

Postnikov asked for a combinatorial proof of the following formula.
\begin{theorem}\label{thm:main}
For every $m\geq1$,
\begin{equation}\label{eq:target}
 N\!\left(\prod_{1\leq i<j\leq m}x_{ij}\right)
 =\prod_{j=1}^{m-1}C_j.
\end{equation}
\end{theorem}

Both formulas follow by counting the same set of integer flows.
We first identify reduction leaves with these flows. We then allow each
vertex a second exit and count flows according to the number of vertices at
which that exit is positive. The main step is a double count of marked units
of flow. It determines the refined counts. Summing them and deleting a
vertex then gives the Catalan factor. Section~\ref{sec:bijection}
makes the cancellations in the double count into an explicit bijection.

\section{Reduction leaves and integer flows}

Write $[m]=\{1,\ldots,m\}$. Represent a monomial by a directed multigraph
$G$ on $[m]$, with an edge $i\to j$ for each occurrence of $x_{ij}$.
All edges point from a smaller vertex to a larger one. Parallel edges are
regarded as distinct, and indegrees count their multiplicities.

Adjoin one sink and an edge from each vertex to that sink; we call these
edges exits. Let $\F(G)$ be the set of assignments
of nonnegative integers $f_e$ to the edges of $G$ and $b_i$ to the exits such
that
\begin{equation}\label{eq:graphflow}
 b_i+\sum_{e:i\to *}f_e
 =\indeg_G(i)+\sum_{e:*\to i}f_e
 \qquad(i\in[m]).
\end{equation}
Thus vertex $i$ supplies $\indeg_G(i)$ units of flow. These sets are finite:
summing~\eqref{eq:graphflow} over $\{1,\ldots,i\}$ bounds every flow
leaving that interval, including its exits.

The proof uses a bijection of Morales and Shi~\cite[Lemma~16]{morales-shi}
for the two graphs produced by a reduction. We give it in the notation
needed here.

\begin{lemma}\label{lem:reduction}
Every complete reduction tree rooted at $G$ has $|\F(G)|$ leaves.
\end{lemma}
\begin{proof}
Select edges $i\to j$ and $j\to k$, where $i<j<k$.
The two children replace this pair by
$(i\to j,i\to k)$ and $(i\to k,j\to k)$, respectively; denote them by
$G_L$ and $G_R$. If the selected edges carry values $p,q$, define a map
\[
 \F(G)\longrightarrow\F(G_L)\sqcupdot\F(G_R)
\]
by the following table. The new values are listed in the order of the edge
pairs just specified.
\[
\begin{array}{c@{\qquad}c@{\qquad}c}
\toprule
\text{Condition}&\text{Image}&\text{Inverse from }(u,v)\\
\midrule
p\geq q&(p-q,q)\text{ in }\F(G_L)&(p,q)=(u+v,v)\\
p<q&(p,q-p-1)\text{ in }\F(G_R)&(p,q)=(u,u+v+1)\\
\bottomrule
\end{array}
\]
Leave all other edge values and all exits unchanged. The indegrees in $G_L$
are the same as in $G$. In $G_R$, the indegree of $j$ decreases by one and
that of $k$ increases by one. Substitution in~\eqref{eq:graphflow} shows that
both maps preserve the required equations. The displayed inverses give a
disjoint bijection.

Every reduction terminates. Indeed, the sum of the edge lengths $j-i$
strictly increases in either child and is bounded by $(m-1)|E(G)|$.
In a terminal graph no vertex has both an incoming and an outgoing edge.
Every edge therefore starts at a vertex of indegree zero, so its flow is
zero. The unique assignment is $f_e=0$ and $b_i=\indeg_G(i)$.
Iterating the bijection proves the lemma.
\end{proof}

Let $K_m$ be the complete directed graph on $[m]$, oriented by the vertex
order, and put $u_m=|\F(K_m)|$. The lemma gives
\begin{equation}\label{eq:um}
 u_m=N\!\left(\prod_{1\leq i<j\leq m}x_{ij}\right).
\end{equation}

\section{Flows with two exits}\label{sec:two-exits}

Alongside the existing exit at each vertex of $K_m$, add a second edge to
the same sink. Write $b_i$ for the flow on the ordinary exit and $r_i$ for
the flow on the added, distinguished exit.
For $0\leq k\leq m-1$, let $\A_{m,k}$ consist of the nonnegative
integer arrays $(f_{ij},b_i,r_i)$ satisfying
\begin{equation}\label{eq:refined}
 b_i+r_i+\sum_{j>i}f_{ij}
 =(i-1)+\sum_{h<i}f_{hi}
 \qquad(i\in[m]),
\end{equation}
with exactly $k$ positive $r_i$. For $F\in\A_{m,k}$, write
\[
 S(F)=\{i:r_i>0\},\qquad a_{m,k}=|\A_{m,k}|.
\]
We call the vertices in $S(F)$ active and the other vertices inactive.
Vertex~1 has only zero outgoing values, so it is always inactive.
In particular, $a_{m,0}=u_m$. We will first determine all the counts
$a_{m,k}$ from $u_m$, then relate their sum to $u_{m+1}$ by deleting a vertex.

In the notation of Morales and Shi~\cite[Theorem~6(i)]{morales-shi},
$a_{m,k}=\Psi_m(k,1,1,1)$: on forgetting the
distinguished exits, the net flow at vertex $i$ is $i-1-r_i$.
The following is their recurrence~(19) for this specialization.
Morales and Shi also suggested interpreting its coefficients by marked
edges~\cite[Section~7.2]{morales-shi}.

\begin{lemma}\label{lem:recurrence}
For $1\leq k\leq m-1$,
\begin{equation}\label{eq:recurrence}
 \binom{k+1}{2}a_{m,k}
 =\binom{m-k+1}{2}a_{m,k-1}.
\end{equation}
\end{lemma}
\begin{proof}
For $F\in\A_{m,k-1}$ and an inactive vertex $i$, equation~\eqref{eq:refined}
becomes
\begin{equation}\label{eq:slots}
 b_i+\sum_{j>i}f_{ij}=\sum_{h<i}(f_{hi}+1).
\end{equation}
At an inactive vertex $i$, label the units on the ordinary exit by
$1,\ldots,b_i$ and those on each outgoing edge $i\to j$ by
$1,\ldots,f_{ij}$. For each incoming edge $h\to i$, use labels
$0,\ldots,f_{hi}$; label~0 represents the extra~1 in~\eqref{eq:slots}.
We call a choice of an edge and one of its labels a slot. A slot always
belongs to a specified flow and inactive vertex.
Sum~\eqref{eq:slots} over all flows $F$ and all their inactive vertices.
We compare the resulting outgoing and incoming slots.

\smallskip\noindent\emph{Ordinary exits.}
An exit slot $s$, where $1\leq s\leq b_i$, determines a new flow by replacing
\[
 (b_i,r_i)=(b_i,0)\quad\text{with}\quad(b_i-s,s).
\]
Mark the newly active vertex $i$. Conversely, merge the distinguished exit
at a marked active vertex into its ordinary exit. Thus ordinary exit slots
are counted by $k\,a_{m,k}$.

\smallskip\noindent\emph{Edges with two inactive endpoints.}
For an edge $i\to j$ whose endpoints are inactive, cancel the outgoing and
incoming slots with the same positive label. The incoming slot with label
zero remains. Each flow in $\A_{m,k-1}$ has $m-k+1$ inactive vertices, so
these remaining slots are counted by
$\binom{m-k+1}{2}a_{m,k-1}$.

\smallskip\noindent\emph{Edges with one active endpoint.}
When $k=1$ there are no such edges in a flow with $k-1$ active vertices.
Suppose $k\geq2$. Fix $i<j$ and a set of $k-2$ active vertices outside
$\{i,j\}$. Consider
the following three families, with this fixed active set outside the pair:
\[
\begin{array}{c@{\qquad}c@{\qquad}c}
\toprule
\text{Family}&\text{Status of }(i,j)&\text{Slot or mark}\\
\midrule
\mathcal U&\text{inactive, active}&1\leq s\leq f_{ij}\\
\mathcal V&\text{active, inactive}&0\leq s'\leq f_{ij}\\
\mathcal W&\text{active, active}&\text{the pair }\{i,j\}\\
\bottomrule
\end{array}
\]
The first two families have $k-1$ active vertices, and the third has $k$.
We give a bijection $\mathcal U\simeq\mathcal V\sqcupdot\mathcal W$.

For an element of $\mathcal U$, put $p=f_{ij}$ and $t=r_j>0$.
Change only $(f_{ij},r_i,r_j)$ and the indicated slot:
\begin{equation}\label{eq:pivot}
 (p,0,t;s)\longmapsto
 \begin{cases}
  (p-s,s,t-s)\in\mathcal W,&s<t,\\
  (p-t,t,0;s-t)\in\mathcal V,&s\geq t.
 \end{cases}
\end{equation}
The values are nonnegative, and the new slot in the second case lies
between $0$ and $p-t$. Both cases preserve~\eqref{eq:refined}. Transferring
$\delta=\min\{s,t\}$ units from $f_{ij}$ and $r_j$ to $r_i$ leaves the total outgoing
flow at $i$ unchanged; at $j$, incoming and outgoing flow both decrease
by $\delta$.

For an element of $\mathcal W$ with values $(p',u,v)$, where $u,v>0$, the
inverse is
\[
 (p'+u,0,u+v;s=u).
\]
For an element of $\mathcal V$ with values $(p',u,0;s')$, where $u>0$, the
inverse is
\[
 (p'+u,0,u;s=s'+u).
\]
These satisfy, respectively, the strict and weak inequalities
in~\eqref{eq:pivot}.

After summing over $i<j$ and the active sets outside the pair, replace
the outgoing slots in $\mathcal U$ by their images in
$\mathcal V\sqcupdot\mathcal W$. The $\mathcal V$ terms cancel the incoming
slots. What remains is one contribution for each flow in $\A_{m,k}$
with a marked pair of active vertices, giving $\binom{k}{2}a_{m,k}$
on the outgoing side.
Edges with two active endpoints contribute no slots, since we summed only
at inactive vertices. We are left with
\[
 k\,a_{m,k}+\binom{k}{2}a_{m,k}
 =\binom{m-k+1}{2}a_{m,k-1},
\]
as required.
\end{proof}

\section{The Catalan product}

Iterating~\eqref{eq:recurrence} gives the Narayana refinement
\begin{equation}\label{eq:narayana}
 a_{m,k}
 =u_m\prod_{\ell=1}^{k}
   \frac{(m-\ell+1)(m-\ell)}{\ell(\ell+1)}
 =\frac{u_m}{m}\binom{m}{k}\binom{m}{k+1}
 \qquad(0\leq k\leq m-1).
\end{equation}
The factors multiplying $u_m$ in the last expression are the Narayana
numbers. This is the specialization recorded in
\cite[Corollary~44]{morales-shi}.
In particular, $a_{m,m-1}=u_m$. Summing and applying Vandermonde's identity,
we obtain
\begin{equation}\label{eq:rowsum}
 \sum_{k=0}^{m-1}a_{m,k}
 =\frac{u_m}{m}\binom{2m}{m-1}=C_m u_m.
\end{equation}
Here the binomial identity counts the choices of $m-1$ elements from two
disjoint $m$-element sets, according to the number chosen from the first.

The identity counted by the next bijection is the case $a=b=c=1$ of
\cite[(17)]{morales-shi}, together with their Corollary~7.
\begin{lemma}\label{lem:deletion}
There is a bijection
\begin{equation}\label{eq:deletion}
 \A_{m+1,m}\simeq\bigsqcup_{k=0}^{m-1}\A_{m,k}.
\end{equation}
\end{lemma}
\begin{proof}
In a flow in $\A_{m+1,m}$, precisely vertices $2,\ldots,m+1$ are active.
Delete vertex~1 and its zero flows, subtract one from each remaining $r_i$,
and relabel $i$ as $i-1$. The outgoing flow and the supply at each remaining
vertex both decrease by one, so~\eqref{eq:refined} holds after relabelling.
The distinguished exits may now be zero. Conversely, insert an initial
vertex with all outgoing values zero and add one to each other distinguished
exit. These operations are inverse.
\end{proof}

\begin{proof}[Proof of Theorem~\ref{thm:main}]
Equations~\eqref{eq:narayana}--\eqref{eq:deletion} give
\[
 u_{m+1}=a_{m+1,m}
 =\sum_{k=0}^{m-1}a_{m,k}=C_m u_m.
\]
Since $u_1=1$, induction and~\eqref{eq:um} prove the formula.
\end{proof}

\begin{proof}[Proof of Theorem~\ref{thm:volume}]
The volume formula of Postnikov--Stanley and Baldoni--Vergne,
in the form stated and proved in~\cite[Theorem~1 and Section~3]{morales-shi},
identifies $\vol(\mathrm{CRY}_{m+1})$ with the
number of nonnegative integer flows having net flow
\[
 \left(0,0,1,\ldots,m-1,-\binom{m}{2}\right).
\]
All edges out of vertex~0 carry zero flow. Deleting that vertex gives
exactly $\F(K_m)$, with vertex $m+1$ as the exit sink. Thus
Theorem~\ref{thm:main} also gives
\[
 \vol(\mathrm{CRY}_{m+1})=u_m=\prod_{j=1}^{m-1}C_j.\qedhere
\]
\end{proof}

\section{A bijection for the recurrence}\label{sec:bijection}

The proof of Lemma~\ref{lem:recurrence} can be made into a single bijection
between two sets of marked flows. Let $\infty$ denote the common sink, and
let $\binom{T}{2}$ denote the set of two-element subsets of $T$.
On the left-hand side below, a pair $\{i,\infty\}$ records one marked
active vertex, and a pair $\{i,j\}$ records two marked active vertices.

\begin{proposition}\label{prop:marked}
For $1\leq k\leq m-1$, there is an explicit bijection
\begin{equation}\label{eq:marked}
\begin{split}
 &\left\{(F,P):F\in\A_{m,k},\quad
        P\in\binom{S(F)\cup\{\infty\}}{2}\right\}\\
 &\hspace{25mm}\simeq
 \left\{(F,P):F\in\A_{m,k-1},\quad
        P\in\binom{[m]\setminus S(F)}{2}\right\}.
\end{split}
\end{equation}
\end{proposition}
\begin{proof}
At each inactive vertex, order the outgoing slots by destination, with the
ordinary exit last, and then by label. Order the incoming slots by source
and then by label. Equation~\eqref{eq:slots} gives lists of equal size;
pair slots of equal rank, meaning the same position in the two lists.
Form these lists using the values of the current flow at each step.

Start with a marked flow on the left of~\eqref{eq:marked}. If the marked
pair is $\{i,\infty\}$, merge $r_i$ into $b_i$ and select the ordinary
exit slot whose label is the former value of $r_i$. If the pair is
$\{i,j\}$ with $i<j$, apply the inverse of the $\mathcal W$ case
in~\eqref{eq:pivot}. In either case we obtain an outgoing slot at an
inactive vertex $i$ in a flow with $k-1$ active vertices.

Match this slot by rank to an incoming slot on an edge $h\to i$.
There are three possibilities:
\begin{enumerate}
\item If $h$ is inactive and the label is zero, stop and mark $\{h,i\}$.
\item If $h$ is inactive and the label is positive, take the outgoing
slot with the same label on $h\to i$.
\item If $h$ is active, apply the inverse of the $\mathcal V$ case
in~\eqref{eq:pivot}, for the pair $h<i$.
\end{enumerate}
The last two cases produce an outgoing slot at the inactive vertex $h$.
Repeat the rank matching there. Each continuation decreases the current
vertex. Vertex~1 has no outgoing slots, so the procedure must stop; it uses
at most $m-1$ rank matches.

To recover the original marked flow, start with the incoming slot labelled~0 on
$i\to j$, where $\{i,j\}$ is the marked inactive pair and $i<j$.
At each step, reverse rank matching at the current vertex $v$, initially
$j$. If it gives an ordinary exit slot, split that exit using its selected
label and mark
$\{v,\infty\}$. If it gives an internal edge
to an inactive vertex, pass to the incoming slot with the same positive
label. If the other endpoint is active, apply~\eqref{eq:pivot}: stop and
mark the active pair in the $\mathcal W$ case, or continue from the
incoming slot supplied by the $\mathcal V$ case.
Every continuation increases the current vertex, so this procedure also
terminates. All its steps reverse the forward construction.
\end{proof}

For each flow on the two sides, the numbers of available marks are
$\binom{k+1}{2}$ and $\binom{m-k+1}{2}$, respectively.
Thus Proposition~\ref{prop:marked} also proves~\eqref{eq:recurrence} directly.

\section*{Acknowledgments}
The author used OpenAI Codex to assist with the preparation of this manuscript.

\end{document}